\documentclass[12pt]{article}
\usepackage{amssymb}
\usepackage{amsmath}
\usepackage{latexsym}
\usepackage{amsthm}
\usepackage{graphics}                     
\usepackage{graphicx}                     
\date{}

\newtheorem{theorem}{Theorem}[section]

\DeclareMathOperator{\STS}{STS}
\newcommand{\lmax}{\ell_{\mathrm{max}}}
\newcommand{\lcmax}{\ell_{\mathrm{cmax}}}

\begin{document}
\title{Good point sequencings of Steiner triple systems\\}
\author{Grahame Erskine and Terry S. Griggs\\
        School of Mathematics and Statistics\\
        The Open University\\
        Walton Hall\\
        Milton Keynes MK7 6AA\\
        UNITED KINGDOM\\
        {\sf grahame.erskine@open.ac.uk}\\
        {\sf terry.griggs@open.ac.uk}}
\maketitle

\begin{abstract}
An $\ell$\textit{-good sequencing} of a Steiner triple system of order $v$, $\STS(v)$, is a permutation of the points of the system such that no $\ell$ consecutive points in the permutation contains a block. It is known that every $\STS(v)$ with $v > 3$ has a 3-good sequencing. It is proved that every $\STS(v)$ with $v \geq 13$ has a 4-good sequencing and every 3-chromatic $\STS(v)$ with $v \geq 15$ has a 5-good sequencing. Computational results for Steiner triple systems of small order are also given.
\end{abstract}

\noindent\textbf{AMS classification:} 05B07.\\
\noindent\textbf{Keywords:} Steiner triple system, good sequencing.

\section{Introduction}\label{sec:Introduction}
The concept of a good sequencing of the points of a Steiner triple system was introduced by Kreher \& Stinson in \cite{KS}. Let $(V,\mathcal{B})$ be a Steiner triple system where $V$ is the set of points and $\mathcal{B}$ is the set of blocks. The \textit{order} of the Steiner triple system is $v = \vert V \vert$, the cardinality of V. Denote such a system by $\STS(v)$. It is well known that Steiner triple systems exist if and only if $v \equiv 1$ or $3 \pmod {6}$. Let $\ell \ge 3$ be an integer. An $\ell$\textit{-good sequencing} of an $\STS(v)$ is a permutation of the points of $V$  such that no $\ell$ consecutive points in the permutation contains a block of the system. Clearly if an $\STS(v)$ has an $(\ell + 1)$-good sequencing then it has an $\ell$-good sequencing, $\ell \ge 3$. In \cite{KS}, the authors proved that every $\STS(v)$ with $v > 3$ has a 3-good sequencing and with $v > 71$ has a 4-good sequencing. The theory of good sequencings was further developed in \cite{SV}. Denote by $\lmax$, the maximum value of $\ell$ for which an $\STS(v)$ can have an $\ell$-good sequencing. We have the following result.
\begin{theorem}[Stinson \& Veitch]\label{thm:max}
~\\
\indent
$\lmax\leq 2s$ if $v = 6s+1, s \ge 2$,
and $\lmax\leq 2s+1$ if $v = 6s+3, s \ge 1$.
\end{theorem}
A further paper on good sequencings is \cite{BE}. The authors prove the following theorem. 
\begin{theorem}[Blackburn \& Etzion]\label{thm:l^4}
~\\\vspace{-12pt} 
\[ For~~v \geq \left(2\ell + 3\binom{\ell - 1}{2}\right)\binom{\ell-1}{2} + \ell = (3\ell^4 - 14\ell^3 + 27\ell^2 - 24\ell + 12)/4, \]
all $\STS(v)$ have an $\ell$-good sequencing.  
\end{theorem}
\noindent The theorem applies also to partial Steiner triple systems.  Substituting the value $\ell = 4$ improves the bound given in \cite{KS} for an $\STS(v)$ to have a 4-good sequencing to $v \geq 55$. The quartic in $\ell$ has since been improved to a quadratic \cite{HC}, but this yields no further improvement for the case of $\ell = 4$, the main focus of this paper.   

In~\cite{BE}, the authors also extend the idea of a sequencing to that of a cyclic sequencing and establish a smaller upper bound on the maximum value of $\ell$ in that case. Denote this upper bound by $\lcmax$.
\begin{theorem}[Blackburn \& Etzion]\label{thm:cycmax}
~\\
\indent
$\lcmax\leq 0.329v +\mathcal{O}(1)$.
\end{theorem}
Clearly if an $\STS(v)$ has an $\ell$-good  cyclic sequencing then it has an $\ell$-good sequencing, $\ell \ge 3$. In \cite{KS}, Kreher \& Stinson give three proofs that every $\STS(v)$ with $v > 3$ has a 3-good sequencing. One of these, which they attribute to Charlie Colbourn, is particularly elegant and is surely a proof from the BOOK \cite{BOOK}. Consider an $\STS(v)$ with $v \geq 7$. Label the points with the numbers $1,2,\ldots,v$ so that the blocks containing $1$ are
\[ \{1,2,v\}, \{1,3,4\},\ldots,\{1,v-2,v-1\}. \]
Then $1,2,\ldots,v-1,v$ is a 3-good sequencing. Unfortunately it is not a 3-good cyclic sequencing because $\{1,2,v\}$ is a block. However the sequence $1,3,2,4,\ldots,v-1,v$ is a cyclic 3-good sequencing unless $\{2,4,5\}$ is a block in which case interchange the two labels 5 and 6.
We state this formally as a theorem.
\begin{theorem}\label{thm:cyc3}
Every $\STS(v)$ with $v > 3$ has a cyclic 3-good sequencing.
\end{theorem}

The main purpose of this paper is to complete the spectrum of $v$ for which there exists an $\STS(v)$ with a 4-good sequencing. These exist for all $v \geq 13$. We are also able to prove that every 3-chromatic $\STS(v)$ with $v \geq 15$ has a 5-good sequencing. These results follow from two constructions which are the subject of Section \ref{sec:Constructions}. The results themselves are proved in Section {\ref{sec:Results}. Finally in Section \ref{sec:Computations} we present some computational results about good sequencing and cyclic good sequencing of Steiner triple systems of small order.      

\section{Constructions}\label{sec:Constructions}
The first construction relates just to 4-good sequencings. In the proof,
and indeed throughout the rest of the paper, it will be convenient to use algebraic (quasigroup) notation, i.e. the third point in the block containing the points $\alpha$ and $\beta$ will be denoted by $\alpha . \beta$.
\begin{theorem}\label{thm:ind}
Let $v\geq 19$ and let $S$ be an $\STS(v)$ with an independent set $I$ of cardinality 8. Then $S$ admits a 4-good sequencing.
\end{theorem}
\begin{proof}
The strategy uses the concept of an independent set, the points of which can be assigned to the sequence in any order since it can contain no block of the $\STS(v)$. We use a greedy algorithm, having first removed the points of the independent set and after carefully choosing the start of the sequence.
   
Let $V$ be the points of $S$ and let $T=V\setminus I$. We begin by placing the first $v-11$ points in the sequence, all chosen from $T$. Choose a point $a\in T$ arbitrarily, and choose a block containing $a$ and wholly contained in $T$. (Since there are only 8 points in $I$, the point $a$ can be in a maximum of 8 blocks containing a point of $I$, so we can certainly do this.) Say the chosen block is $\{a,b,e\}$. Now choose another point $c\in T$ arbitrarily, and a further point $d\in T$ with the restrictions that $d \notin \{a.c,b.c,c.e\}$. Begin the sequence with $a,b,c,d,e$.

Now choose a further point $f \in T$ with the restrictions that $f \notin \{c.d,c.e,d.e\}$. Continue in this manner choosing points from $T$; it is clear that provided we still have at least 4 points in $T$ to choose from, there is an available point to continue the sequence at every stage. Stop the process when there are only 3 points $x,y,z$ left in $T$. (It is possible that some or all of $x,y,z$ could still be adjoined to the sequence but that is not important.) Suppose the last 3 points added to the sequence are $p,q,r$.

Because $a.b=e$, at least one of $x,y,z$ can be adjoined to the start of the sequence. Assume without loss that this is $z$; then adjoin $z$ to the start and $x,y$ to the end of the sequence, leaving 8 spaces for the points of $I$ to be inserted. The partial sequence now is
\[ z,a,b,c,d,e,f,\ldots,p,q,r,*,*,*,*,*,*,*,*,x,y \]
with the asterisks to be replaced by the 8 points of $I$.

For the leftmost asterisk, there are 8 available points of $I$ from which we must exclude $p.q$, $p.r$ and $q.r$. So there are at least 5 available choices and we select a point $s$ from these to continue the sequence. Similarly, for the next point $t$ we have 7 points of $I$ remaining and $q.r$, $q.s$ and $r.s$ need to be excluded. The next point $u$ now has 6 available points remaining but this time only $r.s$ and $r.t$ need to be excluded because $s.t$ cannot be in the independent set $I$. Now the partial sequence is
\[ z,a,b,c,d,e,f,\ldots,p,q,r,s,t,u,*,*,*,*,*,x,y \]
and it remains to replace the last 5 asterisks with the remaining points of $I$. Begin at the rightmost asterisk; there are 5 available points and only $x.y$ must be excluded, so choose point $w$ from the remainder. To continue, there are 4 points available and $x.y$, $w.y$ and $w.x$ must be excluded, so a point $h$ can be chosen. For the next point there are 3 available points but only $h.x$ and $w.x$ are excluded so a point $g$ can be chosen. The partial sequence now is
\[ z,a,b,c,d,e,f,\ldots,p,q,r,s,t,u,*,*,g,h,w,x,y \]
and the remaining 2 points can be chosen in either order from the remaining points of $I$ to complete a 4-good sequencing of $S$.
\end{proof}
The second construction is more general.
\begin{theorem}\label{thm:col}
Let $\ell\geq 3$ and $c=(\ell^2-3\ell+6)/2$. Let $S$ be an $\STS(v)$ which admits a colouring in which every colour class, except possibly one, has at least $c$ points. Then $S$ admits an $\ell$-good sequencing.
\end{theorem}
\begin{proof}
The strategy is to produce a sequence in which the colour classes are consecutive groups of points. Such a group can contain no blocks of the Steiner triple system since a colour class is an independent set. For each colour class we choose the first $\ell-1$ points in the sequence carefully, to avoid creating blocks with points of the previous colour class.

Begin by listing the points of the smallest colour class in any order. Let the last $\ell-1$ points in this sequence be $a_1,a_2,\ldots,a_{\ell-1}$. (If there happen to be fewer than $\ell-1$ points in the smallest colour class, then the following argument goes through with suitable minor amendments.) Choose any other colour class to continue the sequence; we now select the first $\ell-1$ points $b_1,b_2,\ldots,b_{\ell-1}$ to go in the sequence.

There are at least $c$ choices for the first point $b_1$. To avoid creating a block, we must ensure that $b_1$ does not form a block with any two of the points $a_1,a_2,\ldots,a_{\ell-1}$. Provided $c>\binom{\ell-1}{2}$ there is at least one suitable choice of $b_1$. Select $b_1$ freely from the list of suitable choices; now there are at least $c-1$ choices for $b_2$ and we must ensure that $b_2$ does not form a block with any pair from $a_2,\ldots,a_{\ell-1},b_1$. So as before, a maximum of $\binom{\ell-1}{2}$ points are ruled out, and provided $c-1>\binom{\ell-1}{2}$ there will be an available choice for $b_2$. For $b_3$ there are now $c-2$ elements to choose from, but this time only a maximum of $\binom{\ell-1}{2}-1$ are ruled out, since we know that $b_1$ and $b_2$ cannot form a block with any point in the same colour class. So provided $c-2>\binom{\ell-1}{2}-1$ there will be an available choice for $b_3$. From this point there are fewer restrictions; at every step the number of available remaining points in the colour class reduces by 1, but the number of points ruled out reduces more quickly since more of the possible pairs of the previous $\ell-1$ elements in the sequence lie wholly in the colour class. It is easy to see that if $c\geq \binom{\ell-1}{2}+2$ then there are enough choices to be able to select the first $\ell-1$ points from the colour class without creating blocks.

Once $\ell-1$ points have been added, the remaining points of the second colour class can be added to the sequence in any order. Now carry out the same procedure for the third and any subsequent colour classes, each of which has at least $c$ points.
\end{proof}
We note that the argument can be extended to deal with cyclically $\ell$-good sequencings. In its simplest form, if we have one colour class of cardinality at least $c+\ell-1=(\ell^2-\ell+4)/2$. and the remainder, apart from one, at least $c$, then we can use the largest colour class at the end of the sequence and make it join back to the start.

\section{Results}\label{sec:Results}
We are now in a position to prove the main result of this paper.
\begin{theorem}\label{thm:four}
Let $v\geq 13$. Then every $\STS(v)$ admits a 4-good sequencing.
\end{theorem}
\begin{proof}
In \cite{EH}, see also \cite {TRIP}, Erd\H{o}s \& Hajnal proved that every $\STS(v)$ has an independent set of cardinality $\lfloor\sqrt{2v}\rfloor$. Thus every $\STS(v)$ with $v \geq 33$ has an independent set of cardinality 8. By the same result, every STS(31) has an independent set of cardinality 7, leaving 24 other points. The pairs of the independent set occur in $\binom{7}{2}=21$ blocks of the system and hence a further point can be adjoined to the independent set. It is also known that every STS(27) and STS(25) \cite {H} and every STS(21) \cite {FGG2} has an independent set of cardinality 8. Therefore by Theorem \ref{thm:ind}, every $\STS(v)$ with $v \geq 21$ admits a 4-good sequencing.    

To deal with the case $v=19$, use Theorem \ref{thm:col}. In that theorem if $\ell=4$ then $c=5$. Every STS(19) is 3-chromatic \cite{EIGHT} and therefore from Theorem 4.1 of \cite {FGG1} is equitably 3-colourable, i.e. can be coloured with colour classes of cardinalities 7, 6 and 6. The cases $v=15$ and $v=13$ are dealt with by exhibiting a 4-good sequencing for each system. This was done in \cite{KS} where it is also shown that neither the unique STS(7) nor the unique STS(9) admits a 4-good sequencing.
\end{proof}
Theorem \ref{thm:col} can be used to prove further results including the one below.
\begin{theorem}\label{thm:five}
Let $v\geq 15$. Then every 3-chromatic $\STS(v)$ admits a 5-good sequencing.
\end{theorem}
\begin{proof}
In Theorem \ref{thm:col}, if $\ell=5$ then $c=8$. Let $v \geq 31$. Then every $\STS(v)$ has an independent set, and therefore a colour class, $I$, whose cardinality satisfies
\begin{center}
$8 \leq |I| \leq (v+1)/2$ if $v \equiv 3$ or $7 \pmod{12}$ and\\
$8 \leq |I| \leq (v-1)/2$ if $v \equiv 1$ or $9 \pmod{12}.~~~~~$
\end{center}
The lower bound comes from the paper by Erd\H{o}s \& Hajnal \cite{EH} as used in the proof of the previous theorem. The upper bound is due to Sauer \& Sch\"{o}nheim \cite{SS}, see also \cite {TRIP}. In either case there are at least $(v-1)/2 \geq 15$ points in the other two colour classes and hence one of these has cardinality greater than or equal to 8. So by Theorem \ref{thm:col}, the $\STS(v)$ has a 5-good sequencing.

For $v \in \{21,25,27\}$ we use results on 3-chromatic $\STS(v)$. In \cite{HR}, Haddad \& R\"{o}dl proved that if an $\STS(v)$ is 3-chromatic and the cardinalities of the colour classes are $c_1$, $c_2$ and $c_3$ where $c_1 \geq c_2 \geq c_3$ then
\begin{center} 
$v=c_1+c_2+c_3 \geq ((c_1-c_2)^2+(c_2-c_3)^2+(c_3-c_1)^2)/2.$
\end{center}
From Lemma 2.3 of \cite{FGG1} we also have that for $v \geq 9, c_1 \leq (v-1)/2$.

For $v=27$, this gives the possible cardinalities of the colour classes $(c_1,c_2,c_3)$ to be $(12,9,6)$, $(12,8,7)$, $(11,10,6)$, $(11,9,7)$, $(11,8,8)$, $(10,10,7)$, $(10,9,8)$ or $(9,9,9)$ all of which satisfy Theorem \ref{thm:col}.

For $v=25$ the possibilities are $(11,8,6)$, $(11,7,7)$, $(10,10,5)$, $(10,9,6)$, $(10,8,7)$, $(9,9,7)$ or $(9,8,8)$, all of which again satisfy Theorem \ref{thm:col} except $(c_1,c_2,c_3)=(11,7,7)$.
However, in this case it is still possible to construct a 5-good sequencing. Let the three colour classes be $\mathcal{A}=\{A_1,A_2,\ldots,A_{11}\}$, $\mathcal{B}=\{B_1,B_2,\ldots,B_7\}$ and $\mathcal{C}=\{C_1,C_2,\ldots,C_7\}$. The maximum number of blocks of the form $\{B_i,B_j,C_k\}$ is 21 and likewise for blocks of the form $\{B_i,C_j,C_k\}$ but in this case there would be 84 pairs of the form $\{B_i,C_j\}$ whereas there can only be 49. So without loss of generality there exists a block of the form $\{B_i,B_j,A_k\}$, say $\{B_6,B_7,A_1\}$. Construct a 5-good sequencing as follows. Begin with listing the points of the set $\mathcal{B}$ in the order $B_1,B_2,\ldots,B_6,B_7$. Continue with points from the set $\mathcal{C}$ following the procedure in Theorem \ref{thm:col}. This will now be possible because the requirement on the cardinality of the set $\mathcal{C}$ is reduced from 8 to 7 because $B_6.B_7=A_1$. Complete the sequence by adjoining the points of the set $\mathcal{A}$, again using the procedure of Theorem \ref{thm:col}.           

For $v=21$, it was proved in \cite {FGG1} that every 3-chromatic STS(21) has a colouring with colour classes $(c_1,c_2,c_3)=(7,7,7)$ or $(8,7,6)$. Consider first the case where the colouring is equitable and let the three colour classes be $\mathcal{A}=\{A_1,A_2,\ldots,A_7\}$, $\mathcal{B}=\{B_1,B_2,\ldots,B_7\}$ and $\mathcal{C}=\{C_1,C_2,\ldots,C_7\}$. Elementary counting determines that if there are $n$ blocks of the form $\{A_i,A_j,B_k\}$ then there are also $n$ blocks of the form $\{B_i,B_j,C_k\}$ and $\{C_i,C_j,A_k\}$. There are $21-n$ blocks of the form $\{A_i,A_j,C_k\}$, $\{B_i,B_j,A_k\}$ and $\{C_i,C_j,B_k\}$. There are always 7 blocks of the form $\{A_i,B_j,C_k\}$. Suppose that $n=0$ or 21; without loss of generality we may assume the former. Construct a 5-good sequencing as follows. List the points of the set $\mathcal{A}$ in any order. Then adjoin the points from the set $\mathcal{B}$ to the front of the list of $\mathcal{A}$ points and points from the set $\mathcal{C}$ to the back of the list of $\mathcal{A}$ points following the procedure in Theorem \ref{thm:col}. The fact that there are no blocks of the form $\{A_i,A_j,B_k\}$ or $\{C_i,C_j,A_k\}$ again reduces the requirement on the cardinality of the sets $\mathcal{B}$ and $\mathcal{C}$ from 8 to 7 so completion of the sequence is guaranteed.

If $n \notin \{0,21\}$, then there exist distinct points $A_1,A_2,A_6,A_7$ such that there are blocks $\{A_1,A_2,C_k\}$ and $\{A_6,A_7,B_{k'}\}$. List the points of the set $\mathcal{A}$ in the order $A_1,A_2,A_3,A_4,A_5,A_6,A_7$ and proceed as above.

Secondly, consider the case where the STS(21) has a colouring with colour classes $(c_1,c_2,c_3)=(8,7,6)$. Let these be $\mathcal{C}=\{C_1,C_2,\ldots, C_8\}$, $\mathcal{B}=\{B_1,B_2,\ldots,B_7\}$ and $\mathcal{A}=\{A_1,A_2,\ldots,A_6\}$. Consider the pairs $A_i,A_j, 1 \leq i \leq j \leq 6$. If these are all in blocks of the form $\{A_i,A_j,B_k\}$, then any point in the colour class $\mathcal{C}$ can be moved to colour class $\mathcal{A}$ to give an equitable colouring and we proceed as above. Otherwise there is a block say $\{A_1,A_2,C_k\}$. List the points of the set $\mathcal{A}$ in the order $A_1,A_2,A_3,A_4,A_5,A_6$ and adjoin points from the set $\mathcal{B}$ to the front of the list of $\mathcal{A}$ points and points from the set $\mathcal{C}$ to the back of the list of $\mathcal{A}$ points as above. 

Finally, there are 11,084,874,829 non-isomorphic STS(19)s \cite{EIGHT} and we have checked by computer that all have a 5-good sequencing. Sequencings for the 80 non-isomorphic STS(15)s are given in the next section. 
\end{proof}

\section{Computations}\label{sec:Computations}
In this section we present some computational results. In particular we are interested in determining for individual systems the largest value of $\ell$ for which the system has an $\ell$-good sequencing and a cyclic $\ell$-good sequencing. We will call such sequencings best possible.

From Theorem \ref{thm:cyc3}, the unique STS(7) and STS(9) have a cyclic 3-good sequencing. It was shown in \cite{KS} that they do not have a 4-good sequencing.
There are two non-isomorphic STS(13)s. One of these is cyclic and on the set $\mathbb{Z}_{13}$ can be generated from the blocks $\{0,1,4\}$ and $\{0,2,7\}$ by the mapping $i \mapsto i+1$ (mod 13). The other is obtained by replacing the blocks $\{1,2,5\}$, $\{1,3,8\}$, $\{3,5,10\}$ and $\{8,10,2\}$ by the blocks $\{3,8,10\}$, $\{2,5,10\}$, $\{1,2,8\}$ and $\{1,3,5\}$. In both cases the sequence 0,1,2,3,4,5,6,7,8,9,10,11,12 is a cyclic 4-good sequencing; the best possible, see Theorem \ref{thm:max}.

There are 80 non-isomorphic STS(15)s. In \cite{SV}, 4-good sequencings were listed for all of these. Below we give best possible sequencings. Of these, all except for systems 1 to 7, 14 and 16 are cyclically 5-good. The remainder are only 5-good though also cyclically 4-good. The systems are in the standard listing as given in Table 1.28 of \cite{HAND}. It is certainly worthy to note that the 9 exceptional systems which do not have cyclically 5-good sequencings all contain an STS(7) subsystem including all 7 systems which contain 3 or more STS(7) subsystems. This includes of course the projective STS(15) (\#1 in the listing).   

\vspace*{2ex}
\noindent
{\small\setlength{\tabcolsep}{2pt}
\begin{tabular}{llllllll}
1. & 04579aed283b16c & 
2. & 023758419cd6eba & 
3. & 023758419dc6bea & 
4. & 023758419dc6eba\\ 
5. & 073529a6edbc841 & 
6. & 073528b1c9ade46 & 
7. & 0237584d6e9b1ac & 
8. & 037528194ebdc6a\\ 
9. & 057329418eb6dca& 
10. & 053728169be4dca& 
11. & 037528169be4cda& 
12. & 081637a94ceb25d\\
13. & 057328196becd4a& 
14. & 0275384cde9a1b6 & 
15. & 037258194dcbe6a & 
16. & 07352cb19aed846\\
17. & 0a2756e43b198cd & 
18. & 06937421eab5d8c & 
19. & 04926b1c78d3a5e & 
20. & 0714589a6ceb23d\\
21. & 082537c6a9e1b4d & 
22. & 038527b14ae9d6c & 
23. & 052394ade8b617c & 
24. & 084512cb7e9d36a\\
25. & 0145786a2dbe3c9 & 
26. & 04517863aceb92d & 
27. & 0725384cbae916d & 
28. & 045926abcd83e17\\ 
29. & 09746a8c5d12e3b & 
30. & 0275386d19ca4be & 
31. & 07415829abde36c & 
32. & 023954c718de6ab\\ 
33. & 07145829abdc63e & 
34. & 07145829abdc36e & 
35. & 01639a472ced5b8 & 
36. & 01549a682bde3c7\\ 
37. & 05914ca78e26d3b & 
38. & 0425916d83b7ace & 
39. & 07316829adbe45c & 
40. & 092456adbc7318e\\ 
41. & 0467258deba31c9 & 
42. & 0475186dec93ab2 & 
43. & 028697d5bc413ae & 
44. & 05194a36db7c28e\\ 
45. & 0254763e1b9ac8d & 
46. & 04627c1839ea5bd & 
47. & 03716859adb42ec & 
48. & 057298de6c34b1a\\ 
49. & 02457e3619cab8d & 
50. & 071542d3b8ac96e & 
51. & 04591a6e38c7bd2 & 
52. & 017638429ecbd5a\\ 
53. & 017638429ecbd5a & 
54. & 061738492ec5dba & 
55. & 0593261c78da4eb & 
56. & 035294ed68cb17a\\ 
57. & 082391a45db7ec6 & 
58. & 0593261c78da4eb & 
59. & 02495aedc83176b & 
60. & 05841ed9a7b6c32\\ 
61. & 05418a2cbe6d379 & 
62. & 0571483c6e9ad2b & 
63. & 02735b6de8914ac & 
64. & 0258417deba6c93\\ 
65. & 05741d3e9ac68b2 & 
66. & 04726853cae9b1d & 
67. & 0145783ce29d6ab & 
68. & 095246eabc7831d\\ 
69. & 0328647dc9b1a5e & 
70. & 054279abc6d138e & 
71. & 042758c3de91b6a & 
72. & 04725619b3ec8da\\ 
73. & 01457839e2cd6ab & 
74. & 01367852dac94eb & 
75. & 04581263adb79ec & 
76. & 04517a62d3ce9b8\\ 
77. & 0425761c39dea8b & 
78. & 084157bceda2369 & 
79. & 085326de7c9b41a & 
80. & 0732658bcd94e1a\\ 
\end{tabular}
}

\vspace*{2ex}
As stated in the proof of Theorem \ref{thm:five}, we have determined that all STS(19)s have a 5-good sequencing. But the best possible is for an STS(19) to have a cyclic 6-good sequencing. We have not checked all the systems but we do have results for the four cyclic STS(19)s. These are referenced as in \cite{MPR}.
We give the starter blocks on the set $\mathbb{Z}_{19}$ from which the systems can be generated by the mapping $i \mapsto i+1$ (mod 19) as well as a best sequencing of the points. The numbers 10 to 19 are represented by the letters a to i respectively. The sequencings for systems A1 and A2 are cyclically 6-good but those for systems A3 and A4 are only 6-good though also cyclically 5-good. The system A4 belongs to a class of Steiner triple systems called Netto systems. These are one of only two classes of systems in which a permutation group acts 2-homogeneously but not 2-transitively on the points \cite{DDST}, see also \cite{TRIP} and provides further evidence that maximum (cyclic) goodness sequencing is not necessarily associated with structural properties or a high degree of symmetry of the Steiner triple system.

\vspace*{2ex}
\noindent
{\small
\begin{tabular}{lll}
A1: & $\{0,1,4\}, \{0,2,9\}, \{0,5,11\}$. &Sequencing 02468acegi13579bdfh.\\     
A2: & $\{0,1,4\}, \{0,2,12\}, \{0,5,13\}$. &Sequencing 02468acegi13579bdfh.\\     
A3: & $\{0,1,8\}, \{0,2,5\}, \{0,4,10\}$. &Sequencing 013475egb8fhc9d2ia6.\\     
A4: & $\{0,1,8\}, \{0,2,5\}, \{0,4,13\}$. &Sequencing 013457di8bc9fhg2ea6.\\     
\end{tabular}
}

\vspace*{2ex}
For someone with access to a powerful computer system and large amounts of CPU time, an analysis of the best possible sequencings of points for both goodness and cyclic goodness for all STS(19)s might be of interest and provide an addition to the already existing paper on their properties \cite {EIGHT}.

We also have results for STS(21)s. There are seven non-isomorphic cyclic systems, again referenced as in \cite{MPR}. Again for completeness we give starter blocks from which on the set $\mathbb{Z}_{21}$, and together with the block $\{0,7,14\}$, the systems can be generated by the mapping $i \mapsto i+1$ (mod 21). We also give a best possible sequencing of the points with the numbers 10 to 21 being represented by the letters a to k respectively. The sequencings for all seven systems are cyclically 6-good but only that for system C2 is also 7-good, the maximum from Theorem \ref{thm:max}.

\vspace*{2ex}
\noindent
{\small
\begin{tabular}{lll}
C1:& $\{0,1,3\}, \{0,4,12\}, \{0,5,11\}$. &Sequencing 012567ac3j4fkdb8ighe9.\\     
C2:& $\{0,1,3\}, \{0,4,12\}, \{0,5,15\}$. &Sequencing 01hfadj9i5gk6c42b783e.\\     
C3:& $\{0,1,5\}, \{0,2,10\}, \{0,3,9\}$. &Sequencing 01234deacf7hji8596bkg.\\     
C4:& $\{0,1,5\}, \{0,2,10\}, \{0,3,15\}$. &Sequencing 012349ak78jfbich56egd.\\ 
C5:& $\{0,1,5\}, \{0,2,13\}, \{0,3,9\}$. &Sequencing 01234bck7adf86hi59egj.\\     
C6:& $\{0,1,9\}, \{0,2,5\}, \{0,4,10\}$. &Sequencing 0123489afgjhdc675ibke.\\     
C7:& $\{0,1,9\}, \{0,2,5\}, \{0,4,15\}$. &Sequencing 0123489ig5cb7h6aejfkd.\\
\end{tabular}
}

\vspace*{2ex}
We also checked the six known 4-chromatic STS(21)s \cite{FGG1} and found that all are cyclically 6-good but none are 7-good. In addition we generated 1000 STS(21)s at random and checked their good sequencing properties. All were cyclically 6-good and all but one were 7-good; but none were cyclically 7-good. We were unable to find an STS(21) with a cyclically 7-good sequencing, but in view of Theorem \ref{thm:cycmax} such a system may not exist.  However if it does then it would be good to find one.

\end{document}